\documentclass{amsart}

\usepackage{amssymb}
\usepackage{cite}      

\newtheorem{thm}          {Theorem}     [section]
\newtheorem{prop}         {Proposition} [section]
\newtheorem{lemma}        {Lemma}	[section]

\newtheorem{rem}          {Remark}	[section]

\newcommand{\RR}{\mathbb{R}}

\newcommand{\cD}{\mathcal{D}}

\newcommand{\cR}{\mathcal{R}}

\newcommand{\supp}{\operatorname{supp}}

\renewcommand{\div}{\operatorname{div}}

\newcommand{\Jac}[1]{\operatorname{Jac}{(#1)}}

\newcommand{\nd}{\noindent}

\newcommand{\eps}{\varepsilon}

\numberwithin{equation}{section}

\begin{document}

\title{Closure of smooth maps in $W^{1,p}(B^3;S^2)$}

\author{Augusto C.\@ Ponce}

\author{Jean Van Schaftingen}

\address{
Universit\'e catholique de Louvain\hfill\break\indent 
D\'epartement de math\'ematique\hfill\break\indent 
Chemin du Cyclotron 2\hfill\break\indent 
1348 Louvain-la-Neuve, Belgium} 

\email{Augusto.Ponce@uclouvain.be, Jean.VanSchaftingen@uclouvain.be}

\dedicatory{To Jean Mawhin and Patrick Habets with esteem and affection}

\date{\today}

\begin{abstract}
For every $2 < p < 3$, we show that $u \in W^{1,p}(B^3;S^2)$ can be strongly approximated by maps in $C^\infty(\overline B \,\!^3;S^2)$ if, and only if, the distributional Jacobian of $u$ vanishes identically. This result was originally proved by Bethuel-Coron-Demengel-H\'elein, but we present a different strategy which is motivated by the $W^{2,p}$-case.
\end{abstract}



\maketitle


\section{Introduction}\label{sec1}

Let $B^3$ be the unit ball and $S^2$ be the unit sphere of $\RR^3$. Given $1 \leq p < \infty$, consider
\begin{equation}\label{1.1}
W^{1,p}(B^3;S^2) = \Big\{ u \in W^{1,p}(B^3;\RR^3)\ ;\ u(x) \in S^2 \text{ a.e.}  \Big\}.
\end{equation}
Although $W^{1,p}(B^3;S^2)$ is not a vector space, it inherits the usual distance from $W^{1,p}(B^3;\RR^3)$; namely,
\begin{equation}\label{1.2}
\|u-v\|_{W^{1,p}} = \|u-v\|_{L^p} + \|\nabla u- \nabla v\|_{L^p} \quad \forall u,v \in W^{1,p}(B^3;S^2).
\end{equation}

Using standard extension and convolution arguments, it is easy to see that every $u \in W^{1,p}(B^3;S^2)$ can be approximated by maps $\varphi \in C^\infty(\overline B \,\!^3;\RR^3)$ with respect to the $W^{1,p}$-distance.
If we assume in addition that $p > 3$, then by Morrey's estimates such approximation converges uniformly to $u$, and we can thus project this sequence back to $S^2$ to obtain an approximation in $C^\infty(\overline B \,\!^3; S^2)$. Although Morrey's estimates are no longer true in the critical case $p=3$, this argument still works as a consequence of the theory of vanishing mean oscillation (VMO) functions:

\begin{thm}[Schoen-Uhlenbeck~\cite{SchUhl:83}]\label{thm1.0}
Let $p \ge 3$. Then, $C^\infty(\overline{B}\,\!^3;S^2)$ is dense in $W^{1,p}(B^3;S^2)$.
\end{thm}

The reader may wonder what happens if $1 \le p < 3$. It turns out that such conclusion is still true if $1 \le p < 2$, but surprisingly it fails if $2 \le p < 3$:

\begin{thm}[Bethuel-Zheng~\cite{BetZhe:88}]\label{thm1.1}
Let $1 \leq p < 3$. Then, $C^\infty(\overline{B}\,\!^3;S^2)$ is dense in $W^{1,p}(B^3;S^2)$ if, and only if, $1 \leq p < 2$.
\end{thm}

The reason for the lack of density in the case $2 \leq p < 3$ is the existence of ``topological singularities'' of maps in $W^{1,p}(B^3;S^2)$. For instance, given a smooth map $g : S^2 \to S^2$, let
\begin{equation}\label{1.3}
u(x) = g\big( \tfrac{x}{|x|} \big) \quad \forall x \in B^3 \setminus \{0\}.
\end{equation}
In this case, $u \in W^{1,p}(B^3;S^2)$ for every $2 \leq p <3$, but $u$ cannot be strongly approximated by smooth maps $\varphi : \overline{B}\,\!^3 \to S^2$ in $W^{1,p}$ if $\deg{g} \neq 0$. 

Indeed, assume by contradiction that there exists a sequence $(\varphi_n)$ in $C^\infty(B^3; S^2)$ strongly converging to $u$ in $W^{1,p}(B^3; S^2)$. By Fubini's theorem, for a.e.\@ $r > 0$,
\[
\varphi_n \to u \quad \text{strongly in } W^{1,p}(\partial B_r; S^2).
\]
If $2 < p < 3$, then by Morrey's estimates $\varphi_n \to u$ uniformly on $\partial B_r$ for any such $r$ (note that $\partial B_r$ has dimension $2$) and thus 
\begin{equation}\label{1.3a}
\deg{(\varphi_n|_{\partial B_r})} \to \deg{(u|_{\partial B_r})} = \deg{g}.
\end{equation}
Since for every $n \ge 1$, $\deg{(\varphi_n|_{\partial B_r})} = 0$, this would imply that $\deg{g} = 0$, which is a contradiction. When $p=2$, by continuity of the degree under VMO-convergence (see \cite{BreNir:95}) assertion \eqref{1.3a} still holds and we can conclude as before.

\medskip
In the above example, $u$ has a topological singularity at $0$. This raises the question of how to find such singularities for a general map $u \in W^{1,p}(B^3; S^2)$. Their location and strength can be detected using a simple yet powerful tool introduced by Brezis-Coron-Lieb~\cite{BreCorLie:86}: the distributional Jacobian ``Jac''.  

Given $p \geq 2$ and a map $u \in W^{1,p}(B^3;S^2)$, consider the vector field
\begin{equation}\label{1.8}
D(u) = \big( u \cdot u_{x_2} \wedge u_{x_3}, u \cdot u_{x_3} \wedge u_{x_1}, u \cdot u_{x_1} \wedge u_{x_2}  \big),
\end{equation}
where $u_{x_i} \in L^p(B^3; \RR^3)$ denotes the partial derivative of $u$ in the weak sense.
Since $u \in W^{1,2} \cap L^\infty$, we have $D(u) \in L^1(B^3;\RR^3)$. We then define the distributional Jacobian as
\begin{equation}\label{1.9}
\Jac{u} = \frac{1}{3} \div{D(u)} \quad \text{in } \cD'(B^3);
\end{equation}
more precisely,
$$
\langle \Jac{u}, \zeta \rangle = - \frac{1}{3} \int_{B^3} D(u) \cdot \nabla \zeta \quad \forall \zeta \in C_0^\infty(B^3).
$$
For instance, if $u$ is smooth (in which case there are no singularities), then one has
$$
\Jac{u} = u_{x_1} \cdot u_{x_2} \wedge u_{x_3} = 0.
$$
On the other hand, if $u$ is given by \eqref{1.3}, then
$$
\Jac{u} = \frac{4 \pi}{3} (\deg{g}) \, \delta_0,
$$
where $\delta_0$ denotes the Dirac mass at the origin.

\medskip
Since smooth maps are not dense in $W^{1,p}(B^3;S^2)$ when $2 \leq p < 3$, one should be able to identify those maps in $W^{1,p}(B^3;S^2)$ which can be approximated by functions in $C^\infty(\overline{B}\,\!^3;S^2)$. It turns out that the only obstruction of density of smooth maps is of topological nature:

\begin{thm}[Bethuel~\cite{Bet:90}]\label{thm1.2}
Let $u \in W^{1,2}(B^3;S^2)$. Then, there exists a sequence $(\varphi_n) \subset C^\infty(\overline{B}\,\!^3;S^2)$ such that 
\begin{equation}\label{1.4}
\varphi_n \to u \quad \text{strongly in } W^{1,2}
\end{equation}
if, and only if,
\begin{equation}\label{1.5}
\Jac{u} = 0 \quad \text{in } \cD'(B^3).
\end{equation}
\end{thm}

The counterpart of Theorem~\ref{thm1.2} in the case $2 < p < 3$ is the following:

\begin{thm}[Bethuel-Coron-Demengel-H\'elein~\cite{BetCorDemHel:91}]\label{thm1.3}
Let $u \in W^{1,p}(B^3;S^2)$, where $2 < p < 3$. Then, there exists a sequence $(\varphi_n) \subset C^\infty(\overline{B}\,\!^3;S^2)$ such that 
\begin{equation}\label{1.6}
\varphi_n \to u \quad \text{strongly in } W^{1,p}
\end{equation}
if, and only if,
\begin{equation}\label{1.7}
\Jac{u} = 0 \quad \text{in } \cD'(B^3).
\end{equation}
\end{thm}

Although Theorem~\ref{thm1.3} is usually attributed to Bethuel~\cite{Bet:90}, such result was never mentioned in \cite{Bet:90}. Actually, Bethuel's proof of Theorem~\ref{thm1.2} uses a removing dipole technique and strongly relies on the fact that $p = 2$.  The proof of Theorem~\ref{thm1.3}, instead, is based on a different strategy used by Bethuel~\cite{Bet:91} in a subsequent work.

\medskip

More generally, we consider a smooth bounded domain  $\Omega \subset \RR^N$ with $N \geq 2$.
The distributional Jacobian still makes sense for maps in $W^{1,p}(\Omega; S^{N-1})$ as long as $p \ge N-1$. The counterparts of Theorems~\ref{thm1.2} and \ref{thm1.3} are presented in the theorem below: 

\begin{thm}[Bethuel~\cite{Bet:90}, Bethuel-Coron-Demengel-H\'elein~\cite{BetCorDemHel:91}]\label{thm1.4}
\quad\\
Let $u \in W^{1,p}(\Omega;S^{N-1})$, where $N-1 \leq p < N$. Then, there exists a sequence $(\varphi_n) \subset C^\infty(\overline{\Omega};S^{N-1})$ satisfying \eqref{1.6} if, and only if, 
\begin{equation}\label{1.10}
\Jac{u} = 0 \quad \text{in } \cD'(\Omega).
\end{equation}
\end{thm}

In addition, one can estimate the $W^{1,N-1}$-distance between any given map $u \in W^{1, N-1}(\Omega; S^{N-1})$ and the class of smooth maps in terms of $L(u)$, the length of the minimal connection of $u$ (see definition \eqref{2.4} below):

\begin{thm}[Bethuel~\cite{Bet:90}]\label{thm1.4a}
If $u \in W^{1, N-1}(\Omega;S^{N-1})$, then
\begin{equation}\label{1.10a}
\inf{\Big\{ \|\nabla u - \nabla\varphi\|_{L^{N-1}}\ ;\ \varphi \in C^\infty(\overline{\Omega};S^{N-1})  \Big\}} \leq C \big(L(u)\big)^{\frac{1}{N-1}}.
\end{equation}
\end{thm}

\medskip

The main goal of this paper is to use a different strategy from \cite{Bet:90,BetCorDemHel:91} to prove Theorem~\ref{thm1.4} for $N-1 < p < N$. An advantage of our approach is that it can be adapted to higher order Sobolev spaces and in particular to $W^{2,p}$; see \cite{BouPonVan:08}. 
As a by-product we also prove the following new counterpart of Theorem~\ref{thm1.4a} when $N-1 < p < N$:

\begin{thm}\label{thm3.1}
If $N-1 < p < N$, then for every $u \in W^{1,p}(\Omega;S^{N-1})$,
\begin{equation}\label{3.2}
\inf{\Big\{ \|u - \varphi\|_{W^{1,p}}\ ;\ \varphi \in C^\infty(\overline{\Omega};S^{N-1})  \Big\}} \leq C \|\nabla u\|_{L^p(A)},
\end{equation}
for some open set $A \subset \Omega$ such that
\begin{equation}\label{3.3}
|A|^{1/p} \leq C \, L(u) \|\nabla u\|_{L^p(\Omega)}.
\end{equation}
\end{thm}

\bigskip
We now explain the main idea in the proof of Theorem~\ref{thm3.1}. We first cover the domain $\Omega$ with finitely many balls $\big( B_r(x_i) \big)_{i \in I}$ and then we modify $u$ on $B_r(x_i)$ according to whether
\begin{equation}\label{1.11}
\int\limits_{B_{2r}(x_i)} |\nabla u|^p < \lambda r^{N-p} \quad \text{or} \quad \int\limits_{B_{2r}(x_i)} |\nabla u|^p \ge \lambda r^{N-p},
\end{equation}
for some parameter $\lambda > 0$ suitably chosen. In the first case, we call $B_r(x_i)$ a \emph{good ball}, otherwise $B_r(x_i)$ is a \emph{bad ball}. This type of condition was introduced in a remarkable work of Bethuel~\cite{Bet:90}.

If $B_r(x_i)$ is a good ball and $\lambda > 0$ is sufficiently small, then most of the values of $u(B_r(x_i))$ lie in a small geodesic disk of $S^{N-1}$. In this case, a projection into this disk and a convolution allow us to replace $u$ on $B_r(x_i)$ by a smooth map. In contrast, if $B_r(x_i)$ is a bad ball, then $u|_{\partial B_r(x_i)}$ need not be contained in a small geodesic disk, but if the radius $r$ is larger than the length of the minimal connection $L(u)$, we can slightly decrease the radius $r$ if necessary so that $u|_{\partial B_r(x_i)}$ is homotopic to a constant. In this case, using an idea of Bethuel-Zheng~\cite{BetZhe:88}, it is possible to use such homotopy to replace $u$ by a smooth map, while keeping the energy on $B_r(x_i)$ under control. 

The detailed constructions on good and bad balls are presented in Sections~\ref{sec6} and \ref{sec5} below. In the next section, we define the distributional Jacobian for maps in $W^{1,p}(\Omega; S^{N-1})$ with $p \ge N-1$, and we explain some of its main properties. In Sections~\ref{sec4} and \ref{sec3} we prove Theorems~\ref{thm1.4} and \ref{thm3.1}.


\section{The distributional Jacobian}\label{sec2}

Let $N \geq 2$ and $\Omega \subset \RR^N$ be a smooth bounded domain.
Given a map $u \in W^{1,N-1}(\Omega; \RR^N) \cap L^\infty$, we consider the $L^1$-vector field
\begin{equation}\label{2.1}
D(u) = (D_1, \ldots, D_N),
\end{equation}
where
\begin{equation}\label{2.2}
D_j = \det{\big[ u_{x_1}, \ldots, u_{x_{j-1}}, u, u_{x_{j+1}}, \ldots, u_{x_N} \big]}.
\end{equation}
We then associate to the map $u$ the distribution
\begin{equation}\label{2.3}
\Jac{u} = \frac{1}{N} \div{D(u)};
\end{equation}
more precisely,
$$
\langle \Jac{u}, \zeta \rangle = - \frac{1}{N} \sum_{j = 1}^N{ \int_\Omega D_j \zeta_{x_j}} \quad \forall \zeta \in C_0^\infty(\Omega).
$$
Given $u \in W^{1, N-1}(\Omega; S^{N-1})$, we define the length of the minimal connection of $u$ as
\begin{equation}\label{2.4}
L(u) = \frac{1}{\omega_N} \sup_{
\substack{
\zeta \in C_0^\infty(\Omega)\\
\|\nabla\zeta\|_{L^\infty} \leq 1
}
}{\langle \Jac{u}, \zeta \rangle},
\end{equation}
where $\omega_N$ denotes the measure of the unit ball $B_1 \subset \RR^N$.
The reason for calling $L(u)$ the length of the minimal connection of $u$ comes from the geometric meaning of $L(u)$ (see equations \eqref{2.10} and \eqref{2.13} below).
If $u$ is smooth, then 
$$
\Jac{u} = \det{\big[ u_{x_1}, \ldots, u_{x_N} \big]} = 0.
$$
More generally, if $u$ is smooth except at finitely many points $a_1, \ldots, a_k \in \Omega$, then (see e.g.\@ \cite{BreCorLie:86})
\begin{equation}\label{2.8}
\Jac{u} = \omega_N \sum_{i=1}^k{d_i \delta_{a_i}} \quad \text{in } \cD'(\Omega),
\end{equation}
where $d_i = \deg{(u,a_i)}$ denotes the degree of $u$ with respect to any small sphere centered at $a_i$. Since we are not making any additional assumption about $u$ on $\partial\Omega$, it may happen that $\sum\limits_{i=1}^k{d_i} \neq 0$. However, by using points from $\partial\Omega$ one can always rewrite \eqref{2.8} as
\begin{equation}\label{2.9}
\Jac{u} = \omega_N \sum_{i=1}^{\tilde k}{(\delta_{p_i} - \delta_{n_i} )} \quad \text{in } \cD'(\Omega),
\end{equation}
where $p_1, \ldots, p_{\tilde k}, n_1, \ldots, n_{\tilde k} \in \overline\Omega$ (note that points on $\partial\Omega$ are harmless from the point of view of test functions with compact support in $\Omega$). 
In particular, one always has $L(u) \le \sum\limits_{i=1}^{\tilde k}{|p_i - n_i|}$. Brezis-Coron-Lieb~\cite{BreCorLie:86} proved that these points can be chosen and rearranged so that
\begin{equation}\label{2.10}
L(u) = \sum_{i=1}^{\tilde k}{|p_i - n_i|}.
\end{equation}

For a general map $u \in W^{1,N-1}(\Omega;S^{N-1})$, not necessarily with finitely many singularities, one has the following characterizations of $\Jac{u}$ and $L(u)$:

\begin{thm}[Bourgain-Brezis-Mironescu~\cite{BouBreMir:04}]\label{thm2.2}
Given $u \in W^{1,N-1}(\Omega;S^{N-1})$,\\ there exist sequences of points $(p_i), (n_i) \subset \overline\Omega$ such that 
$\sum\limits_{i=1}^\infty{|p_i - n_i|} < \infty$ and
\begin{equation}\label{2.12}
\Jac{u} = \omega_N \sum_{i=1}^{\infty}{(\delta_{p_i} - \delta_{n_i} )} \quad \text{in } \cD'(\Omega).
\end{equation}
Moreover,
\begin{equation}\label{2.13}
L(u) = \inf{\left\{ \sum_{i=1}^{\infty}{|p_i - n_i|}\ ;\ (p_i), (n_i) \subset \overline\Omega \text{ satisfy \eqref{2.12}} \right\}}.
\end{equation}
\end{thm}

In contrast with the case of finitely many singularities, the infimum in \eqref{2.13} need not be achieved in general; see \cite{Pon:03d}.

\medskip
We end this section by showing the well-known fact that $L(u)$ is continuous with respect to the strong convergence in $W^{1,N-1}(\Omega; S^{N-1})$:

\begin{prop}\label{prop2.1}
Let $(u_n) \subset W^{1,N-1}(\Omega; S^{N-1})$ be a sequence such that $u_n \to u$ in $W^{1,N-1}$.
Then,
\begin{equation}
L(u_n)\to L(u).
\end{equation}
\end{prop}

\begin{proof}
Note that for every $u, v \in W^{1,N-1}(\Omega; \RR^N) \cap L^\infty$ we have
\begin{equation}\label{2.5}
\big| \langle \Jac{u} - \Jac{v}, \zeta \rangle \big| \leq \|D(u) - D(v)\|_{L^1}\|\nabla\zeta\|_{L^\infty} \quad \forall \zeta \in C_0^\infty(\Omega).
\end{equation}
Thus, a standard argument gives
\begin{equation}\label{2.6}
\big| L(u) - L(v) \big| \leq \|D(u) - D(v)\|_{L^1}.
\end{equation}
If $(u_n)$ is a sequence converging strongly to $u$ in $W^{1, N-1}$, then by dominated convergence $D(u_n) \to D(u)$ in $L^1$ and the conclusion holds.
\end{proof}


\section{A Fubini-type argument}\label{sec7}

In Sections~\ref{sec6}--\ref{sec5} we present the main ingredients in the proof of Theorem~\ref{thm3.1}. The construction in those sections rely on an argument based on Fubini's theorem which we shall explain below. But first, given $1 \leq p < \infty$, let us introduce the following class of functions:
\begin{equation}\label{2.11}
\cR^{1,p}(\Omega) = 
\left\{ v \in W^{1,p}(\Omega;S^{N-1}) 
\left|
\begin{aligned}
&\text{there exist } a_1, \ldots, a_k \in \Omega \text{ such that}\\
&\text{$v$ is smooth in $\overline\Omega \setminus \{a_1, \ldots, a_k\}$}
\end{aligned}
\right.
\right\}.
\end{equation}
For later use, given $v \in \cR^{1,p}(\Omega)$, we denote by $S(v)$ the set of points of $\Omega$ where $v$ is not smooth (by definition this set is finite).

\bigskip

As we have already explained, smooth maps are not dense in $W^{1,p}(\Omega;S^{N-1})$ if $N-1 \leq p < N$. However,

\begin{thm}[Bethuel-Zheng~\cite{BetZhe:88}]\label{thm2.1}
If $N-1 \leq p < N$, then $\cR^{1,p}(\Omega)$ is dense in $W^{1,p}(\Omega;S^{N-1})$.
\end{thm}

This result is particularly useful since it reduces the problem of studying maps in $W^{1,p}(\Omega; S^{N-1})$ into a problem where all maps have finitely many singularities. This is for instance one of the main ingredients in the proof of Theorem~\ref{thm2.2} above. For the sake of Theorem~\ref{thm3.1}, one could avoid Theorem~\ref{thm2.1}, but the proof becomes less transparent.

\bigskip
We show in this section that if $v \in \cR^{1,p}(\Omega)$ and if $B_r(x_0)$ is a sufficiently large ball contained in $\Omega$, then it is possible to find a sphere $\partial B_s(x_0)$ such that $v|_{\partial B_s(x_0)}$ is homotopic to a constant:

\begin{lemma}\label{lemma7.1}
Let $N-1 \leq p < N$. Given $v \in \cR^{1,p}(\Omega)$, let $r > 0$ be such that
\begin{equation}\label{7.1}
r > 4L(v).
\end{equation}
Then, for every $x_0 \in \Omega$ with $B_{2r}(x_0) \subset \Omega$ there exists $s \in \big(\frac{3 r}{2},2r \big)$ such that
\begin{equation}\label{7.2}
\deg{(v|_{\partial B_s(x_0)})} = 0 \quad \text{and} \quad \|\nabla v\|_{L^p({\partial B_s(x_0)})} \leq \frac{C}{r^{1/p}} \|\nabla v\|_{L^p(B_{2r}(x_0))};
\end{equation}
moreover, there exists $\psi \in C_0^\infty(B_{2r}(x_0))$ such that $\psi = 1$ on $B_s(x_0)$ and
\begin{equation}\label{7.2a}
\big| \langle \Jac{v}, (1-\psi) \zeta\rangle \big| \le L(v) \|\nabla\zeta\|_{L^\infty(\Omega)} \quad \forall \zeta \in C_0^\infty(\Omega).
\end{equation}
\end{lemma}

\begin{proof}
By scaling and translation we can assume that $r = 1$ and $x_0 = 0$.
Let $p_1, \ldots, p_{\tilde k}$ and $n_1, \ldots, n_{\tilde k}$ in $\overline\Omega$ be such that
\begin{equation}\label{6.3}
\Jac{v} = \omega_N \sum_{i=1}^{\tilde k}{(\delta_{p_i} - \delta_{n_i} )} \quad \text{in } \cD'(\Omega)
\end{equation}
and
\begin{equation}\label{6.4}
L(v) = \sum_{i=1}^{\tilde k}{|p_i - n_i|}.
\end{equation}
Denote by $[p_i,n_i]$ be the segment joining $p_i$ to $n_i$. Let
$$
T = \left\{ t \in \left(\tfrac{3}{2},2 \right) ; \ \partial B_t \cap [p_i,n_i] = \emptyset \quad \forall i \in \big\{ 1, \ldots, \tilde k \big\} \right\}.
$$
Since $L(v) < 1/4$, it follows from the area formula that $|T| > 1/4$. On the other hand, by Fubini's theorem,
$$
\int_T dt \int_{\partial B_t} |\nabla v|^p \leq \int_{B_2} |\nabla v|^p.
$$
Thus, there exists $s \in T$ such that
\begin{equation}\label{6.5}
\int_{\partial B_{s}} |\nabla v|^p \leq 4 \int_{B_2} |\nabla v|^p.
\end{equation}
Moreover, since $s \in T$, the number of points $p_i$ and $n_i$ inside the ball $B_s$ (including multiplicities) are equal; thus, $\deg{(v|_{\partial B_{s}})} = 0$. It remains to show \eqref{7.2a}. To prove this we use the fact that $\partial B_s$ does not intersect any of the segments $[p_i, n_i]$. Thus, for some $\eps > 0$ small, the annulus $B_{s+\eps} \setminus B_s$ does not intersect any of those segments. Let $\psi \in C_0^\infty(B_2)$ be such that $\psi = 1$ in $B_s$. Denoting by $I \subset \{1, \ldots, \tilde k\}$ the set of indices such that $[p_i, n_i]$ is not in $B_{s+\eps}$, we then have
\[
\langle \Jac{v}, (1-\psi) \zeta \rangle = \sum_{i \in I}{\bigl[\zeta(p_i) - \zeta(n_i) \bigr]} \quad \forall \zeta \in C_0^\infty(\Omega)
\]
and thus
\[
\bigl|\langle \Jac{v}, (1-\psi) \zeta \rangle \bigr| \le L(v) \|\nabla\zeta\|_{L^\infty(\Omega)}  \quad \forall \zeta \in C_0^\infty(\Omega).
\qedhere
\]
\end{proof}


\section{Replacing $u$ on bad balls}\label{sec6}

Given $\lambda > 0$ and a ball $B_r(x_0)$ such that $B_{2r}(x_0) \subset \Omega$, we say that $B_r(x_0)$ is a bad ball
for a map $v \in W^{1,p}(\Omega; S^{N-1})$ if
\begin{equation}\label{6.0}
\int\limits_{B_{2r}(x_0)} |\nabla v|^p \ge \lambda r^{N-p}.
\end{equation}

We explain below how to replace $v$ by a smooth map on bad balls. This construction is possible if the radius $r$ is large enough compared to the length of the minimal connection $L(v)$. At this stage, the choice of the parameter $\lambda >0$ plays no role whatsoever in the proof.

\begin{prop}\label{prop6.1}
Let $N-1 < p < N$. If $B_r(x_0)$ is a bad ball for $v \in \cR^{1,p}(\Omega)$ and if
\begin{equation}\label{6.1}
r > 4 L(v),
\end{equation}
then one can find $w \in \cR^{1,p}(\Omega)$ such that
\begin{itemize}
\item[$(B_1)$] $w$ is smooth in $B_r(x_0)$;
\item[$(B_2)$] $w = v$ in $\Omega \setminus B_{2r}(x_0)$;
\item[$(B_3)$] $L(w) \leq L(v)$ and $S(w) \subset S(v)$;
\item[$(B_4)$] $\|w-v\|_{L^p(\Omega)} \leq C r \|\nabla w - \nabla v\|_{L^p(B_{2r}(x_0))}$;
\item[$(B_5)$] $\|\nabla w- \nabla v\|_{L^p(\Omega)} \leq C \|\nabla v\|_{L^p(B_{2r}(x_0))}.$
\end{itemize}
\end{prop}

\begin{proof}
We shall use a strategy similar to the proof of \cite[Lemma~1]{Bet:91}. \\
We may assume that $\|\nabla v\|_{L^p(B_{2r}(x_0))} > 0$, for otherwise $v$ is constant in $B_{2r}(x_0)$ and there is nothing to prove. By scaling and translation, we may also suppose that $r = 1$ and $x_0 = 0$.
Since $r$ satisfies \eqref{6.1}, by Lemma~\ref{lemma7.1} there exists $s \in \big(\frac{3}{2}, 2\big)$ such that
\begin{equation}\label{6.2}
\deg{(v|_{\partial B_s})} = 0 \quad \text{and} \quad \|\nabla v\|_{L^p({\partial B_s})} \leq C \|\nabla v\|_{L^p(B_2)}.
\end{equation}
Let
\[
\tilde v(x) = 
\begin{cases}
v(x)	& \text{if $x \in \Omega \setminus B_s$,}\\
v\big( s \tfrac{x}{|x|} \big) & \text{if $x \in B_s$.}
\end{cases}
\]
Then, $\tilde v \in \cR^{1,p}(\Omega)$,
$\tilde v$ is continuous in $B_s \setminus \{0\}$ 
and, by the choice of $s$,
\[
\|\nabla \tilde v\|_{L^p(B_s)} = C \|\nabla v\|_{L^p(\partial B_s)} \le C \|\nabla v\|_{L^p(B_2)}.
\]
Using the triangle inequality, we then get
\begin{equation}\label{6.2a}
\begin{split}
\|\nabla \tilde v - \nabla v\|_{L^p(\Omega)} 
& = \|\nabla \tilde v - \nabla v\|_{L^p(B_s)}\\
& \leq \|\nabla \tilde v \|_{L^p(B_s)} + \| \nabla v\|_{L^p(B_s)} \le \widetilde C \|\nabla v\|_{L^p(B_2)}.
\end{split}
\end{equation}
Note that $\tilde v$ is continuous in a neighborhood of $\partial B_s$ but $\tilde v$ is not necessarily smooth there. By convolution and projection we may modify $\tilde v$ to make it smooth near $\partial B_s$. For this reason, we shall henceforth suppose that we do have $\tilde v \in \cR^{1,p}(\Omega)$. 

By \eqref{6.2a}, the map $\tilde v$ satisfies $(B_5)$ but $\tilde v$ need not satisfy $(B_1)$ because of its possible singularity at $0$.
We now use the fact that $\deg{(v|_{\partial B_s})} = 0$ to remove that singularity. Indeed, by the Hopf theorem, $v|_{\partial B_s}$ is homotopic to a constant. One can thus find a continuous homotopy $H : [0,1] \times \partial B_s \to S^{N-1}$ such that 
$H(t,\cdot) = p_0$ if $0 \le t \le \frac{1}{3}$ for some $p_0 \in S^{N-1}$ and $H(t, \cdot) = v|_{\partial B_s}$ if $\frac{2}{3} \le t \le 1$. Making a convolution of $H$ and projecting the resulting map back to $S^{N-1}$, one can even assume that $H$ belongs to $C^\infty\big([0,1] \times \partial B_s; S^{N-1} \big)$ (recall that $H$ was just assumed to be continuous and needed not be even in $W^{1, p}$). Since $H$ is constant on $[0, \frac{1}{3}]$, for every $0 < \eps < t$, the map
\[
w_\eps(x) = 
\begin{cases}
\tilde v(x)	& \text{if $x \in \Omega \setminus B_\eps$,}\\
H\big( \frac{|x|}{\eps}, x \big) & \text{if $x \in B_\eps$,}
\end{cases}
\]
belongs to $W^{1,p}(\Omega; S^{N-1})$ and is continuous in $B_s$. Since $w_\eps \to \tilde v$ strongly in $W^{1, p}$ as $\eps \to 0$, we can take $\eps > 0$ sufficiently small so that 
\begin{equation}\label{6.2b}
\|\nabla w_\eps - \nabla \tilde v\|_{L^p(\Omega)} \le \|\nabla v\|_{L^p(B_2)}.
\end{equation}
Combining \eqref{6.2a}--\eqref{6.2b} we deduce that $w_\eps$ also satisfies $(B_5)$. Since $w_\eps = v$ outside the ball $B_2$, by Poincar\'e's inequality,
\[
\|w-v\|_{L^p(\Omega)} = \|w-v\|_{L^p(B_2)} \leq C \|\nabla w - \nabla v\|_{L^p(B_2)}
\]
and thus $(B_4)$ also holds. In order to check property $(B_3)$ we can use \eqref{7.2a}. Indeed, since $w_\eps$ is smooth on $B_s$, $\Jac{w_\eps} = 0$ on $B_s$. Thus, if $\psi \in C_0^\infty(B_2)$ denotes the function given by Lemma~\ref{lemma7.1}, then
\[
\langle \Jac{w_\eps}, \zeta\rangle = \langle \Jac{w_\eps}, \psi\zeta\rangle + \langle \Jac{w_\eps}, (1-\psi)\zeta\rangle = \langle \Jac{w_\eps}, (1-\psi)\zeta\rangle,
\]
for every $\zeta \in C_0^\infty(\Omega)$. Since $v = w_\eps$ on $\Omega \setminus B_s$ and $\psi = 1$ on $B_s$, it follows that
\[
\langle \Jac{w_\eps}, \zeta\rangle = \langle \Jac{v}, (1-\psi)\zeta\rangle.
\]
Taking the supremum over all test functions $\zeta$ with $\|\nabla\zeta\|_{L^\infty} \le 1$, we deduce from \eqref{7.2a} that $L(w_\eps) \le L(v)$, which is the desired inequality.
\end{proof}

\begin{rem}\label{rem6.1}
\rm
Strictly speaking, in the previous proof we have not used the fact that $B_r(x_0)$ was a bad ball, but we do it now. In fact, since $B_r(x_0)$ is a bad ball,
\begin{align*}
|B_{2r}(x_0)|^{1/p} = \bigl(\omega_N (2r)^N \bigr)^{1/p} = \Bigl(\frac{2^N\omega_N}{\lambda} \Bigr)^{1/p} \, r (\lambda r^{N-p})^{1/p} \le C r \|\nabla v\|_{L^p(B_{2r}(x_0))},
\end{align*}
where the constant $C > 0$ depends on the choice of $\lambda$. We can thus rewrite property $(B_5)$ in the way it will be used in the proof of Theorem~\ref{thm3.1}:
\begin{itemize}
\item[$(B_5')$] $\|\nabla w- \nabla v\|_{L^p(\Omega)} \leq C \|\nabla v\|_{L^p(A)}$,
for some open set $A \subset B_{2r}(x_0)$ such that $|A|^{1/p} \leq C r \|\nabla v\|_{L^p(B_{2r}(x_0))}$.
\end{itemize}
\end{rem}


\section{Replacing $u$ on good balls}\label{sec5}

Given $\lambda > 0$ and a ball $B_r(x_0)$ such that $B_{2r}(x_0) \subset \Omega$, we say that $B_r(x_0)$ is a good ball
for a map $v \in W^{1,p}(\Omega; S^{N-1})$ if
\begin{equation}\label{5.1}
\int\limits_{B_{2r}(x_0)} |\nabla v|^p < \lambda r^{N-p}.
\end{equation}

In this section we explain how to replace $v$ by a smooth map on good balls. This construction strongly relies on a suitable choice of the parameter $\lambda$.

\begin{prop}\label{prop5.1}
Let $N-1 < p < N$. There exists $\lambda = \lambda(N,p) > 0$ such that if $B_r(x_0)$ is a good ball for $v \in \cR^{1,p}(\Omega)$ and if
\begin{equation}\label{5.1a}
r > 4 L(v),
\end{equation}
then one can find $w \in \cR^{1,p}(\Omega)$ such that
\begin{itemize}
\item[$(G_1)$] $w$ is smooth in $B_r(x_0)$;
\item[$(G_2)$] $w = v$ in $\Omega \setminus B_{2r}(x_0)$;
\item[$(G_3)$] $L(w) \leq L(v)$ and $S(w) \subset S(v)$;
\item[$(G_4)$] $\|w-v\|_{L^p(\Omega)} \leq C r \|\nabla w - \nabla v\|_{L^p(B_{2r}(x_0))}$;
\item[$(G_5)$] $\|\nabla w- \nabla v\|_{L^p(\Omega)} \leq C \|\nabla v\|_{L^p(A)}$,
for some open set $A \subset B_{2r}(x_0)$ such that $|A|^{1/p} \leq C r \|\nabla v\|_{L^p(B_{2r}(x_0))}$.
\end{itemize}
\end{prop}

\begin{proof}
We can assume that
\begin{equation}\label{5.3}
\|\nabla v\|_{L^p(B_{2r}(x_0))} > 0,
\end{equation}
for otherwise $v$ is constant in $B_r(x_0)$ and the conclusion is obvious.
By scaling and translation, we may also assume that $r = 1$ and $x_0 = 0$.
Since $r$ satisfies \eqref{5.1a}, by Lemma~\ref{lemma7.1} there exists $s \in \big(\frac{3}{2}, 2\big)$ such that
\begin{equation}\label{6.2x}
\deg{(v|_{\partial B_s})} = 0 \quad \text{and} \quad \|\nabla v\|_{L^p({\partial B_s})} \leq C \|\nabla v\|_{L^p(B_2)}.
\end{equation}
Since $p > N - 1$, it follows from Morrey's estimates that $v|_{\partial B_s}$ is a continuous function and there exists $\lambda_1 > 0$ (depending only on $N$ and $p$) such that if 
\begin{equation}\label{5.6}
\|\nabla v \|_{L^p(\partial B_{s})} \leq \lambda_1,
\end{equation} 
then $v(\partial B_{s})$ is a subset of $S^{N-1}$ of diameter at most $1/3$. We then choose $\lambda$ so that
$$
C\lambda^{1/p} =  \lambda_1,
$$
where $C$ is the constant in \eqref{6.2x}.
We denote by $D_{1/3}(\xi_0)$ a closed geodesic disk of $S^{N-1}$ of radius $1/3$ containing $v(\partial B_{s})$ and centered at $\xi_0$. Let $\Phi : S^{N-1} \to S^{N-1}$ be a smooth function such that $\Phi(x)  = x$, $\forall x \in D_{2/3}(\xi_0)$, $\|\Phi'\|_{L^\infty}  \leq 2$ and
\begin{equation}\label{5.8}
\Phi(S^{N-1})  \subset D_1(\xi_0).
\end{equation}
Let 
\begin{equation}\label{5.10}
\tilde v =
\begin{cases}
v		& \text{in } \Omega \setminus B_{s},\\
\Phi \circ v	& \text{in } B_{s}.
\end{cases}
\end{equation}
Then, $\tilde v \in \cR^{1,p}(\Omega)$ and
\begin{equation}\label{5.11}
\int_\Omega |\nabla v - \nabla \tilde v|^p = \int_{B_{s}} |1 - \Phi'(v)|^p |\nabla v|^p 
\leq C \int_{U} |\nabla v|^p,
\end{equation}
where 
\begin{equation}\label{5.12}
A = \Big\{ x \in B_s \setminus S(v)\ ;\ v(x) \not\in D_{2/3}(\xi_0)  \Big\}.
\end{equation}
Since $v$ is continuous on $B_s \setminus S(v)$, $A$ is an open set. We now show that
\begin{equation}\label{5.12a}
|A|^{1/p} \leq C \|\nabla v\|_{L^p(B_2)}.
\end{equation}
For this purpose, consider the function
$$
f(x) = \big[3 \,d(v(x),\xi_0) - 1 \big]^+ \quad \forall x \in B_{s},
$$
where $d$ denotes the geodesic distance in $S^{N-1}$. Note that 
$$
f \geq 1 \text{ on $A$,}  \quad f = 0 \text{ on $\partial B_{s}$} \quad  \text{and} \quad |\nabla f| \leq 3|\nabla v| \text{ a.e.}
$$
Thus, by Chebyshev's and Poincar\'e's inequalities,
\begin{equation}\label{5.13}
|A| \leq \int_{B_{s}} |f|^p \leq C \int_{B_{s}} |\nabla f|^p \leq 3^p C \int_{B_{s}} |\nabla v|^p \le 3^p C \int_{B_2} |\nabla v|^p,
\end{equation}
which gives \eqref{5.12a}.
Although $\tilde v$ need not be continuous in $B_1$, its image is contained in a geodesic disk of $S^{N-1}$. A standard argument allows us to replace $\Phi \circ v$ by a function which is smooth in $B_1$.\\
We present a detailed proof for the convenience of the reader. We first take a family of nonnegative smooth mollifiers $(\rho_\eps) \subset C_0^\infty(\RR^N)$ and $\zeta \in C_0^\infty(B_{3/2})$ such that $\supp{\zeta} \subset B_{3/2}$ and $\zeta = 1$  on $B_1$. Consider
\begin{equation}\label{5.14}
v_\eps = (1-\zeta) \Phi(v) + \zeta \big[\rho_\eps * \Phi(v)\big] \quad \text{in } B_s.
\end{equation}
Denote by $V$ the convex hull in $\RR^N$ of the geodesic disk $D_1(\xi_0)$. By \eqref{5.8} we have
$$
\Phi(v(x)) \in V \quad \text{and} \quad \big[\rho_\eps * \Phi(v)\big](x) \in V \quad \forall x \in B_s.
$$
Thus,
\begin{equation*}
v_\eps(x) \in V \quad \forall x \in B_s.
\end{equation*}
On the other hand, we have $|y| \geq 1/2$ for every $y \in V$. Therefore,
\begin{equation}\label{5.15}
|v_\eps(x)| \geq \frac{1}{2} \quad \forall x \in B_s.
\end{equation}
In particular, 
\begin{equation}\label{5.16}
\frac{v_\eps}{|v_\eps|} \to \Phi (v) \quad \text{in } W^{1,p}.
\end{equation}
Take $\eps > 0$ sufficiently small so that
\begin{equation}\label{5.17}
\int_{B_s} \Big|\nabla \Bigl(\frac{v_\eps}{|v_\eps|}\Bigr) - \nabla \big(\Phi (v) \big) \Big|^p \leq \int_A |\nabla v|^p.
\end{equation}
If the integral in the right-hand side vanishes, take $A \subset B_1$ to be any open set of measure at most $\|\nabla v\|^p_{L^p(B_2)}$ for which the right-hand side is not zero; this is possible in view of \eqref{5.3}.
Let $w$ be the function given by
\begin{equation}\label{5.18}
w =
\begin{cases}
v		& \text{in } \Omega \setminus B_{s},\\
\dfrac{v_\eps}{|v_\eps|}	& \text{in } B_s.
\end{cases}
\end{equation}
This function satisfies $(B_5)$ and, by Poincar\'e's inequality, also satisfies $(B_4)$. The proof of the inequality $L(w) \le L(v)$ follows the same lines as in the previous lemma. Indeed, since the image of $\frac{v_\eps}{|v_\eps|}$ is contained in a small geodesic disk, all singularities of $w$ in $B_s$ have degree zero. Thus, $\Jac{w} = 0$ in $B_s$. Thus, if $\psi \in C_0^\infty(B_2)$ denotes the function given by Lemma~\ref{lemma7.1}, then for every $\zeta \in C_0^\infty(\Omega)$,
\[
\langle \Jac{w_\eps}, \zeta\rangle = \langle \Jac{v}, (1-\psi)\zeta\rangle,
\]
which implies $L(w) \le L(v)$.
\end{proof}


\section{Replacing $u$ on balls near the boundary}\label{sec8}

The reader probably have noticed that even though the constructions performed on bad balls and on good balls are different, the conclusions of Propositions~\ref{prop6.1} and \ref{prop5.1} ---taking into account Remark~\ref{rem6.1}--- are the same.
The goal of this section is twofold: to merge both statements and to take into account the possibility of performing the same construction on balls which need not be entirely contained in $\Omega$.

Note that the underlying notions of good balls and bad balls can be adapted to balls which are not entirely contained in $\Omega$ in a straightforward way. 
Actually, there are essentially two types of balls $B_r(x_0)$ one should really take care of: those such that $B_{2r}(x_0) \subset \Omega$, which have been studied in Sections~\ref{sec6} and \ref{sec5} above, and those such that $x_0 \in \partial\Omega$, which will be our main concern in the proof below. Indeed, the general construction can be always reduced to one of these types.

\begin{prop}\label{prop8.1}
Let $N-1 < p < N$. There exists $\delta = \delta(\Omega) > 0$ such that if $v \in \cR^{1,p}(\Omega)$ and if
\begin{equation}\label{8.2}
\delta > r > 4 L(v),
\end{equation}
then for every $x_0 \in \overline\Omega$ there exists $w \in \cR^{1,p}(\Omega)$ such that
\begin{itemize}
\item[$(M_1)$] $w$ is smooth in $B_r(x_0) \cap \overline\Omega$;
\item[$(M_2)$] $w = v$ in $\Omega \setminus B_{8r}(x_0)$;
\item[$(M_3)$] $L(w) \leq L(v)$ and $S(w) \subset S(v)$;
\item[$(M_4)$] $\|w-v\|_{L^p(\Omega)} \leq C r \|\nabla w - \nabla v\|_{L^p(B_{8r}(x_0) \cap \Omega) }$;
\item[$(M_5)$] $\|\nabla w- \nabla v\|_{L^p(\Omega)} \leq C \|\nabla v\|_{L^p(A)}$,
for some open set $A \subset B_{8r}(x_0)  \cap \Omega$ such that $|A|^{1/p} \leq C r \|\nabla v\|_{L^p(B_{8r}(x_0) \cap \Omega)}$.
\end{itemize}
\end{prop}

\begin{proof}
If $B_{2r}(x_0) \subset \Omega$, the conclusion follows from Proposition~\ref{prop6.1} (and Remark~\ref{rem6.1}) or from Proposition~\ref{prop5.1} depending on whether $B_r(x_0)$ is a bad ball or a good ball. We may then restrict ourselves to the case where $B_{2r}(x_0) \cap \partial\Omega \neq \emptyset$. We shall reduce the problem to a situation where the ball is centered at some point of $\partial\Omega$. Indeed, since $B_{2r}(x_0) \cap \partial\Omega \neq \emptyset$, there exists $y_0 \in \partial\Omega$ such that $|y_0 - x_0| < 2r$ and thus
\[
B_r(x_0) \subset B_{3r}(y_0) \quad \text{and} \quad B_{6r}(y_0) \subset B_{8r}(x_0).
\]
It thus suffices to construct a map $w \in \cR^{1,p}(\Omega)$ such that $w$ is smooth in $B_{3r}(y_0) \cap \overline\Omega$, $w = v$ in $\Omega \setminus B_{6r}(y_0)$ and satisfying $(M_3)$--$(M_5)$.

In what follows, we assume that $B_{6r}(y_0) \cap \partial\Omega$ is flat and thus $B_{6r}(y_0) \cap \Omega$ coincides with a half-ball. By a translation and a scaling argument, we may suppose that $y_0 = 0$ and $r = \frac{1}{3}$. By the Fubini-type argument of Lemma~\ref{lemma7.1}, one finds $s \in \big( \frac{3}{2}, 2 \big)$ such that
\begin{equation}\label{8.2a}
\|\nabla v\|_{L^p(\partial B_s \cap \Omega)} \le C \| \nabla v\|_{L^p(B_2)}
\end{equation}
and $\partial B_s$ does not intersect any of the segments $[p_i, n_i]$, where the points $p_i$ and $n_i$ denote the singularities of $v$ arranged so as to satisfy \eqref{6.4}.

If $B_1$ is a \emph{bad ball} for $v$, in the sense that
\[
\int\limits_{B_2 \cap \Omega} |\nabla v|^p \ge \lambda
\]
for some parameter $\lambda > 0$ to be chosen later on, then we proceed as in the proof of Proposition~\ref{prop6.1} and define
\[
\tilde v(x) = 
\begin{cases}
v(x)	& \text{if $x \in \overline\Omega \setminus B_s$,}\\
v\big( s \tfrac{x}{|x|} \big) & \text{if $x \in B_s \cap \overline\Omega$,}
\end{cases}
\]
which is continuous except possibly at $0$ and satisfies
\begin{equation}\label{8.3}
\| \nabla \tilde v - \nabla v\|_{L^p(\Omega)} \le C \|\nabla v\|_{L^p(B_2 \cap \Omega)}.
\end{equation}
Since $u |_{\partial B_s \cap \Omega}$ is necessarily homotopic to a constant map (recall that $\partial B_s \cap \Omega$ is a half-sphere, which is topologically trivial), one can remove that singularity at $0$ as in Proposition~\ref{prop6.1} without losing property \eqref{8.3}. Thus, we get a map $w \in \cR^{1, p}(\Omega)$ which is now smooth on $B_s \cap \overline\Omega$ and
\[
\| \nabla w - \nabla v\|_{L^p(\Omega)} \le C \|\nabla v\|_{L^p(B_2 \cap \Omega)}.
\]
Since $B_1$ was assumed to be a bad ball, as in Remark~\ref{rem6.1} we have
\[
|B_2 \cap \Omega|^{1/p} \leq C_\lambda \|\nabla v\|_{L^p(B_2 \cap \Omega)},
\]
and thus $w$ satisfies $(M_5)$ with $A = B_{6r}(y_0)$ (which corresponds to $B_2$ after translation and scaling). Property $(M_4)$ just follows from Poincar\'e's inequality. Finally, since $\partial B_s$ does not intersect any of the segments $[p_i, n_i]$, one deduces that $L(w) \le L(v)$. Thus, $w$ satisfies all the required properties.

On the other hand, if $B_1$ is a \emph{good ball} for $v$, in the sense that
\[
\int\limits_{B_2 \cap \Omega} |\nabla v|^p < \lambda,
\]
then in view of \eqref{8.2a}, $\|\nabla v\|_{L^p(\partial B_s \cap \Omega)} < C \lambda^{1/p}$. Therefore, by Morrey's estimates we can fix some $\lambda  > 0$ sufficiently small (depending on $N$ and $p$) so that $v(\partial B_s \cap \Omega)$ is contained in a small geodesic disk of $S^{N-1}$. One can then proceed exactly as in the proof of Proposition~\ref{prop5.1} by taking a family of convolutions $(\rho_\eps)$ supported in $B_1 \cap \Omega$; this way the function $v_\eps$ remains well-defined and the conclusion follows.

We now deal with the case where $B_{6r}(y_0) \cap \partial\Omega$ is not necessarily flat.
By choosing $\delta > 0$ sufficiently small (depending on $\Omega$) it is possible to find a diffeomorphism $\Phi$ such that the image of $B_{3r}(y_0) \cap \Omega$ is contained in the half-ball $B_{3r}^+$ and the image of $B_{6r}(y_0) \cap \Omega$ contains the half-ball $B_{6r}^+$. We can then apply the previous construction to the map $v \circ \Phi^{-1}$.
The proof of the proposition is complete.
\end{proof}


\section{Proof of Theorem~\ref{thm3.1}}\label{sec4}

Let us assume momentarily that we have proved \eqref{3.2} for maps $u \in \cR^{1,p}(\Omega)$. We show that this implies a similar estimate for every $u \in W^{1,p}(\Omega; S^{N-1})$. Indeed, given $u \in W^{1,p}(\Omega; S^{N-1})$ we consider two separate case, whether $L(u) = 0$ or $L(u)>0$. We first assume that $L(u) = 0$. Taking a sequence $(u_n) \subset \cR^{1,p}(\Omega)$ such that $u_n \to u$ strongly in $W^{1,p}$, then by continuity of the length of the minimal connection,
\[
L(u_n) \to L(u) = 0.
\]
By \eqref{3.2} applied to $u_n$ and Lebesgue's dominated convergence theorem,
\[
\inf{\Big\{ \|u_n - \varphi\|_{W^{1,p}}\ ;\ \varphi \in C^\infty(\overline{\Omega};S^{N-1})  \Big\}} \to 0.
\]
Therefore, there exists a sequence $(\varphi_n) \subset  C^\infty(\overline{\Omega};S^{N-1})$ such that $\varphi_n \to u$ strongly in $W^{1,p}$. Hence, $u$ satisfies \eqref{3.2} with $A = \emptyset$. On the other hand, if $L(u) > 0$, then we first take an open set $A_1 \subset \Omega$ such that 
\[
\|\nabla u\|_{L^p(A_1)} > 0 \quad \text{and} \quad |A_1|^{1/p} \le L(u) \|\nabla u\|_{L^p(\Omega)}
\]
and then, by Theorem~\ref{thm2.2}, one can choose $v \in \cR^{1,p}(\Omega)$ such that
\[
\|u - v\|_{W^{1,p}(\Omega)} \le \|\nabla u\|_{L^p(A_1)}. 
\]
We may also assume that $v$ satisfies
\[
L(v) \le 2 L(u) \quad \text{and} \quad \|\nabla v\|_{L^p(\Omega)} \le 2 \|\nabla u\|_{L^p(\Omega)}.
\]
Since by assumption estimate \eqref{3.2} holds for $v$, there exists $\varphi \in  C^\infty(\overline{\Omega};S^{N-1})$ such that
\[
\|v - \varphi\|_{W^{1,p}(\Omega)} \le 2 C \|\nabla v\|_{L^p(A_2)},
\]
where $A_2 \subset \Omega$ is an open set satisfying $|A_2| \le C L(v) \|\nabla v\|_{L^p(\Omega)}$.
We then have
\begin{equation*}
\begin{split}
\|u - \varphi\|_{W^{1, p}} 
& \le \|u - v\|_{W^{1, p}} + \|v - \varphi\|_{W^{1, p}}\\
& \le \|\nabla u\|_{L^p(A_1)} + C\|\nabla v\|_{L^p(A_2)}\\
& \le \|\nabla u\|_{L^p(A_1)} + C \bigl( \|\nabla u\|_{L^p(A_2)} + \|\nabla u - \nabla v\|_{L^p(A_2)}\bigr)\\
& \le \|\nabla u\|_{L^p(A_1)} + C \bigl( \|\nabla u\|_{L^p(A_2)} + \|\nabla u\|_{L^p(A_1)}\bigr)\\
& \le (1 + 2 C) \|\nabla u\|_{L^p(A_1 \cup A_2)},
\end{split}
\end{equation*}
where 
\begin{equation*}
\begin{split}
|A_1 \cup A_2|^{1/p} 
& \le |A_1|^{1/p} + |A_2|^{1/p}\\
& \le L(u) \|\nabla u\|_{L^p(\Omega)} + C L(v) \|\nabla v\|_{L^p(\Omega)}\\
& \le L(u) \|\nabla u\|_{L^p(\Omega)} + 4C L(u) \|\nabla u\|_{L^p(\Omega)}\\
& = (1 + 4C) L(u) \|\nabla u\|_{L^p(\Omega)}.
\end{split}
\end{equation*}
Thus, $u$ also satisfies an estimate of the type \eqref{3.2}. 

\bigskip

In view of the above it suffices to establish \eqref{3.2} for maps $u \in \cR^{1,p}(\Omega)$. Let $\delta > 0$ be the quantity given by Proposition~\ref{prop8.1}, depending only on $\Omega$. We consider two separate cases:

\medskip
\nd
\textit{Case 1.}\/ $4 L(u) < \delta$.

\smallskip

Let $r > 0$ be such that $4 L(u) < r < \delta$. 
We can cover $\Omega$ with balls $(B_{r}(x_i))_{i\in I}$ in such a way that, for every $i \in I$, $x_i \in \overline\Omega$ and each ball $B_{8r}(x_i)$ intersects at most $\theta$ balls $B_{8r}(x_j)$, where $\theta$ depends only on the dimension $N$.
We can thus split the set of indices $I$ as $I=I_1\cup \cdots \cup I_{\theta+1}$ so that for any $i=1,\ldots, \theta+1$ and any distinct indices $j_1,j_2\in I_i$ we have $B_{8r}(x_{j_1})\cap B_{8r}(x_{j_2})=\emptyset$.

Starting from $u_0 = u$, we construct maps $u_1, \ldots, u_{\theta+1} \in W^{1,p}(\Omega; S^{N-1})$ inductively as follows. Given $k \geq 0$ and $u_{k}$ we apply Proposition~\ref{prop8.1} to the map $u_{k}$ and to each ball $B_r(x_i)$ with $i \in I_{k+1}$ until we exhaust $I_{k+1}$; denote by $u_{k+1}$ the map obtained by this procedure. 
Since the balls $\big(B_ {8r}(x_i)\big)_{i \in I_{k+1}}$ are disjoint, by properties $(M_4)$--$(M_5)$ we have
\begin{align}
\label{3.0x}
\|u_{k+1} - u_{k}\|_{L^p(\Omega)} 
& \leq C r \|\nabla u_{k+1} - \nabla u_k\|_{L^p(\Omega)}\\
\label{3.1x}
\|\nabla u_{k+1} - \nabla u_{k}\|_{L^p(\Omega)} 
& \leq C \| \nabla u_{k} \|_{L^p(E_k)}
\end{align}
for some open set $E_k \subset \Omega$ such that $|E_k|^{1/p} \leq C r \|\nabla u_k\|_{L^{p}(\Omega)}$; $E_k$ is the union of all sets $A$ arising from Proposition~\ref{prop8.1}.\\
By induction, it follows from \eqref{3.0x}--\eqref{3.1x} that for every $k = 1, \ldots, \theta + 1$ we have
\begin{align}
\label{3.2x}
\|u_k - u\|_{L^p(\Omega)} 
& \leq C_k r \|\nabla u\|_{L^{2p}(\Omega)}\\
\label{3.3x}
\|\nabla u_{k} - \nabla u\|_{L^p(\Omega)} 
& \leq C_k \| \nabla u \|_{L^p(F_k)}
\end{align}
where $F_k = \bigcup\limits_{j=0}^{k-1} E_j$.
We first prove \eqref{3.3x}. Since the conclusion is clear if $k=1$, we may assume that \eqref{3.2x} holds for some $k \ge 1$. We then have
\begin{align*}
\|\nabla u_{k+1} - \nabla u\|_{L^p(\Omega)}
& \leq \|\nabla u_{k+1} - \nabla u_k\|_{L^p(\Omega)} + \|\nabla u_{k} - \nabla u\|_{L^p(\Omega)} \\
\text{\small(by \eqref{3.1x})} \quad
& \leq C \|\nabla u_k\|_{L^p(E_k)} + \|\nabla u_{k} - \nabla u\|_{L^p(\Omega)} \\
\text{\small(by triangle inequality)} \quad
& \leq C \|\nabla u\|_{L^p(E_k)} + (1 + C)\|\nabla u_{k} - \nabla u\|_{L^p(\Omega)} \\
\text{\small(by induction)} \quad
& \leq C \|\nabla u\|_{L^p(E_k)} + (1 + C)C_k\|\nabla u\|_{L^p(F_{k-1})}\\
& \leq \big[C + (1 + C)C_k \big] \|\nabla u\|_{L^p(F_{k-1} \cup E_k)}.
\end{align*}
This establishes \eqref{3.3x}. Combining \eqref{3.0x} and \eqref{3.3x}, one gets \eqref{3.2x}. Note in addition that the set $F_k$ satisfies 
$|F_k|^{1/p} \leq C_k r \|\nabla u\|_{L^{p}(\Omega)}$. Indeed, proceeding by induction we have
\begin{align*}
|F_{k+1}|^{1/p} & \le |F_k|^{1/p} + |E_k|^{1/p} \\
\text{\small(by estimate on $|E_k|$)} \quad
& \leq  |F_k|^{1/p} + C r \|\nabla u_k\|_{L^{p}(\Omega)}\\
\text{\small(by induction)} \quad
& \leq C_k r \|\nabla u\|_{L^{p}(\Omega)} + C r \|\nabla u_k\|_{L^{p}(\Omega)}\\
\text{\small(by triangle inequality)} \quad
& \leq C_k r \|\nabla u\|_{L^{p}(\Omega)} + C r \big( \|\nabla u\|_{L^{p}(\Omega)} + \|\nabla u_k - \nabla u\|_{L^{p}(\Omega)} \big)\\
\text{\small(by \eqref{3.3x})} \quad
& \leq C_k r \|\nabla u\|_{L^{p}(\Omega)} + C r \big( \|\nabla u\|_{L^{p}(\Omega)} + C_k \| \nabla u \|_{L^p(F_k)} \big)\\
& \leq (C_k + C (1 + C_k)) r \|\nabla u\|_{L^{p}(\Omega)},
\end{align*}
which gives the estimate for the sets $|F_k|$.
Since the balls $(B_{r}(x_i))_{i\in I}$ cover $\Omega$ and we have swept away all the singularities of $u$ from of these balls, the map $u_{\theta+1}$ is smooth. 
We have thus obtained for every $r > 4 L(u)$ a map $\varphi_r \in C^\infty(\overline\Omega; S^{N-1})$, namely $u_{\theta+1}$, such that
\begin{align}
\label{3.5x}
\|\varphi_r - u\|_{L^p(\Omega)} 
& \leq C r \|\nabla u\|_{L^{2p}(\Omega)}\\
\label{3.6x}
\|\nabla \varphi_r - \nabla u\|_{L^p(\Omega)} 
& \leq C \| \nabla u \|_{L^p(A_r)}
\end{align}
where $A_r \subset \Omega$ is an open set such that $|A_r|^{1/p} \leq C r \|\nabla u\|_{L^{p}(\Omega)}$. 
If $L(u) = 0$, it follows from dominated convergence that $\varphi_r \to u$ strongly in $W^{1,p}$ and thus \eqref{3.2} holds with $A = \emptyset$. 
Otherwise, $L(u) > 0$, in which case we can take $r \approx 4 L(u)$.

\medskip
\nd
\textit{Case 2.}\/ $4 L(u) \geq \delta$.

\smallskip
We show the conclusion holds by taking $A = \Omega$.
Indeed, by an easy variant of Poincar\'e's inequality, there exists $\alpha_u \in S^{N-1}$ such that
\begin{equation}\label{4.2}
\|u - \alpha_u\|_{L^p(\Omega)} \leq C \|\nabla u \|_{L^p(\Omega)},
\end{equation}
where the constant $C > 0$ does not depend on $u$; thus,
$$
\|u - \alpha_u\|_{W^{1,p}(\Omega)} \leq (1 + C) \|\nabla u\|_{L^p(\Omega)}.
$$
On the other hand, by H\"older's inequality,
\begin{equation*}
L(u) \leq |\Omega|^{1- \frac{N-1}{p}} \|\nabla u\|_{L^p(\Omega)}^{N-1}.
\end{equation*}
Thus,
\begin{equation}\label{4.1}
L(u) \|\nabla u\|_{L^p(\Omega)} \geq \frac{\big(L(u)\big)^{\frac{N}{N-1}}}{|\Omega|^{\frac{1}{N-1} - \frac{1}{p}}} \geq \frac{(\delta/4)^{\frac{N}{N-1}}}{|\Omega|^{\frac{1}{N-1}}} |\Omega|^{1/p} = C_0 |\Omega|^{1/p},
\end{equation}
where $C_0 > 0$ is a constant depending on $N$ and $\Omega$.

\medskip
In both cases, we have obtained estimate \eqref{3.2}. The proof of the theorem is complete.
\qed


\section{Proof of Theorem~\ref{thm1.4}}\label{sec3}

The implication $(\Leftarrow)$ follows from Theorem~\ref{thm3.1} if $N-1 < p < N$ or from Theorem~\ref{thm1.4a} if $p = N - 1$. To prove the converse,
let $(\varphi_n) \subset C^\infty(\overline{\Omega};S^{N-1})$ be a sequence such that 
\begin{equation*}
\varphi_n \to u \quad \text{strongly in } W^{1,p}.
\end{equation*}
For every $n \geq 1$, we have $\Jac{\varphi_n} = 0$; thus, $L(\varphi_n) = 0$. In view of Proposition~\ref{prop2.1}, this implies 
$L(u) = 0$ or, equivalently, $\Jac{u} = 0$ in $\cD'(\Omega)$. \qed


\section*{Acknowledgments}

The authors would like to thank P.~Bousquet for interesting discussions and the Department of Mathematics of the Universit\'e Fran\c cois Rabelais (Tours, France), where part of this work was carried out. The second author was supported by the Fonds de la Recherche scientifique--FNRS and by the Fonds sp\'ecial de Recherche.




\begin{thebibliography}{10}
\expandafter\ifx\csname url\endcsname\relax
  \def\url#1{\texttt{#1}}\fi
\expandafter\ifx\csname urlprefix\endcsname\relax\def\urlprefix{URL }\fi
\providecommand{\selectlanguage}[1]{\relax}
\providecommand{\eprint}[2][]{\url{#2}}

\bibitem{Bet:90}
F.~Bethuel, \emph{A characterization of maps in {$H\sp 1(B\sp 3,S\sp 2)$} which
  can be approximated by smooth maps}. Ann. Inst. H. Poincar\'e Anal. Non
  Lin\'eaire \textbf{7} (1990), 269--286.

\bibitem{Bet:91}
F.~Bethuel, \emph{The approximation problem for {S}obolev maps between two
  manifolds}. Acta Math. \textbf{167} (1991), 153--206.

\bibitem{BetCorDemHel:91}
F.~Bethuel, J.-M. Coron, F.~Demengel, and F.~H\'elein, \emph{A cohomological
  criterion for density of smooth maps in {S}obolev spaces between two
  manifolds}. Nematics (Orsay, 1990), NATO Adv. Sci. Inst. Ser. C Math. Phys.
  Sci. \textbf{332} (1991), 15--23.

\bibitem{BetZhe:88}
F.~Bethuel and X.~M. Zheng, \emph{Density of smooth functions between two
  manifolds in {S}obolev spaces}. J. Funct. Anal. \textbf{80} (1988), 60--75.

\bibitem{BouBreMir:04}
J.~Bourgain, H.~Brezis, and P.~Mironescu, \emph{{$H\sp {1/2}$} maps with values
  into the circle: minimal connections, lifting, and the {G}inzburg-{L}andau
  equation}. Publ. Math. Inst. Hautes \'Etudes Sci. \textbf{99} (2004), 1--115.

\bibitem{BouPonVan:08}
P.~Bousquet, A.~C. Ponce, and J.~Van~Schaftingen, \emph{A case of density in
  {$W\sp {2,p}(M;N)$}}. C. R. Math. Acad. Sci. Paris \textbf{346} (2008),
  735--740.

\bibitem{BreCorLie:86}
H.~Brezis, J.-M. Coron, and E.~H. Lieb, \emph{Harmonic maps with defects}.
  Comm. Math. Phys. \textbf{107} (1986), 649--705.

\bibitem{BreNir:95}
H.~Brezis and L.~Nirenberg, \emph{Degree theory and {BMO}. {I}. {C}ompact
  manifolds without boundaries}. Selecta Math. (N.S.) \textbf{1} (1995),
  197--263.

\bibitem{Pon:03d}
A.~C. Ponce, \emph{On the distributions of the form {$\sum\sb i(\delta\sb {p\sb
  i}-\delta\sb {n\sb i})$}}. J. Funct. Anal. \textbf{210} (2004), 391--435.

\bibitem{SchUhl:83}
R.~Schoen and K.~Uhlenbeck, \emph{Boundary regularity and the {D}irichlet
  problem for harmonic maps}. J. Differential Geom. \textbf{18} (1983),
  253--268.

\end{thebibliography}
\end{document}